\documentclass[a4paper]{amsart}
\usepackage{amssymb}
\usepackage{tikz}
\usepackage{epic,curves,mfpic}
\usepackage{hyperref}

\newtheorem{theo}{Theorem}[section]
\newtheorem{lemma}[theo]{Lemma}

\newtheorem{conjec}[theo]{Conjecture}

\theoremstyle{definition}

\newcommand{\Z}{\mathbb{Z}}

\newcommand{\B}{\mathcal{B}}
\newcommand{\co}{\mathcal{O}}

\theoremstyle{remark}

\begin{document}

\title{Almost Invariant Sets }
\author{M.J. Dunwoody  }

% LaTex only know AMS SubjClass2000

\begin{abstract}
A short proof of a conjecture of Kropholler is given.   This gives a relative version of Stallings' Theorem
on the structure of groups with more than one end.   A generalisation of the Almost Stability Theorem is also obtained, that gives information
about the structure of the Sageev cubing.
\end{abstract}
\maketitle

\begin{section} {Introduction}
Let $G$ be a group.    A  subset $A$ of $G$ is said to be  {\it almost invariant} if the symmetric difference $A + Ag$ is finite for every
$g \in G$.   In addition $A$ is said to be {\it proper} if both $A$ and $A^* = G -A$ are infinite.   The group $G$ is said to have more than one end
if it has a proper almost invariant subset.

\begin{theo}\label{SC} 
A group $G$ contains a proper almost invariant subset (i.e. it has more than one end) if and only if it has a non-trivial action
on a tree with finite edge stabilizers.
\end{theo}
This result was proved by Stallings \cite {[St]} for finitely generated groups and was generalized to all groups by Dicks and Dunwoody \cite {[DD]}.
The action of a group $G$ on a tree is {\it trivial} if there is a vertex that is fixed by all of $G$.  Every group has a trivial action
on a tree.

Let $T$ be a tree with directed edge set $ET$.      If $e$ is a directed edge, then let $\bar e$ denote $e$ with the reverse orientation.
If $e, f$ are distinct directed edges then  write $e >f$  if the smallest subtree of $T$ containing $e$ and $f$ is as below.

\begin{figure}[htbp]
\centering
\begin{tikzpicture}[scale=.5]
\draw (-.5, 2)-- (1,2)--(2,1)--(3,2)--(4,1)--(5,2)--(6,1)--(7,2)--(8,1)--(9,2)--(10.5,2)  ;
\draw (.5,2) node {$>$} ;
\draw (9.5,2) node {$>$} ;
\draw (.5,2.1) node [above] {$e$} ;
\draw (9.5,2.1) node [above] {$f$} ;
\end{tikzpicture}
\end{figure}

Suppose the group $G$ acts on $T$.   We say that $g$ {\it shifts } $e$ if either $e >ge$ or $ge > e$.
If for some $e \in ET$ and some $g \in G$,  $g$ shifts $e$,   then $G$ acts non-trivially on a tree $T_e$ obtained by contracting
all edges of $T$ not in the orbit of $e$ or $\bar e$.   In this action there is just one orbit of edge pairs.
Bass-Serre theory tells us that either $G = G_u*_{G_e} G_v$ where $u, v$ are the vertices of $e$ and they are in different orbits  in the contracted tree $T_e$, or $G$ is the 
HNN-group $G = G_u *_{G_e}$ if $u,v$ are in the same $G$-orbit.   
If either case occurs we say that $G$ splits over $G_e$.

If there is no edge $e$ that is shifted by any $g \in G$, (and $G$ acts without involutions, i.e. there is no $g \in G$ such that $ge = \bar e$)
then $G$ must fix a vertex or an end of $T$.   If the action is non-trivial, it fixes an end of $T$, i.e.
$G$ is a union of an ascending sequence of vertex stabilizers,  $G = \bigcup  G_{v_n}$, where $v_1, v_2, \dots   $ is a sequence of 
adjacent vertices and $G_{v_1} \leq G_{v_2}\leq \dots $ and $G \not= G_{v_n} $ for any $n$.

Thus Theorem ~\ref{SC} could be restated as 

\begin{theo} [\cite {[St]}, \cite{[DD]}] 
A group $G$ contains a proper almost invariant subset (i.e. it has more than one end) if and only if it splits over a finite subgroup or
it is countably infinite and locally finite.
\end{theo}
The if part of the theorem is fairly easy to prove.  We now prove  a stronger version of the if part, following \cite {[D]}.

Let $H$ be a subgroup of $G$.   A subset $A$ is  $H$-{\it finite} if $A$ is contained in finitely many right $H$-cosets, i.e. for some finite
set $F$,   $A \subseteq HF$. 
A subgroup $K$ is $H$-finite if and only if $H\cap K$ has finite index in $K$.
   Let $T$ be a $G$-tree and suppose there is an edge $e$ and  vertex $v$.

We say that $e$ {\it  points at} $v$ if there is a subtree of $T$ as below.
We write $e \rightarrow v$.

\begin{figure}[htbp]
\centering
\begin{tikzpicture}[scale=.5]
\draw (-.5, 2)-- (1,2)--(2,1)--(3,2)--(4,1)--(5,2)--(6,1)--(7,2)--(8,1)--(9,2)--(10.5,2)  ;
\draw (.5,2) node {$>$} ;
%\draw (9.5,2) node {$>$} ;
\draw (.5,2.1) node [above] {$e$} ;
\draw (10.5,2.1) node [above] {$v$} ;
\draw (10.5,2) node {$\bullet$} ;

\end{tikzpicture}
\end{figure}

Let $G[e, v] = \{g \in G | e \rightarrow gv\}$.

If $h \in G$, then $G[e, v]h = G[e,h^{-1}v]$,  since if $e\rightarrow gv,  e \rightarrow gh (h^{-1}v)$.

It follows from this that If $K = G_v$, then $G[e,v]K =G[e,v]$.
Also if $H = G_e$, then $HG[e,v] = G[e,v]$.

If $v = \iota e $, then $G_e = H \leq K = G_v$ and if $A = G[e, \iota e]$, then $A= HAK$.

\begin{figure}[htbp]
\centering
\begin{tikzpicture}[scale=.5]
\draw (-2, 2)-- (2,2) ;
\draw (.0,2) node {$>$} ;
\draw (2,2) node {$\bullet$} ;
\draw (-2,2) node {$\bullet$} ;

\draw (.5,2.1) node [above] {$e$} ;
\draw (-2,2.1) node [above] {$v$} ;
\end{tikzpicture}
\end{figure}

Consider the set $Ax,  x\in G$.  If $g \in A, gx \notin A$ , then $e \rightarrow gv,  \bar e \rightarrow gx v$.
This means that $e$ is on the directed  path joining $gxv$ and $gv$.  This happens if and only
if $g^{-1}e$ is on the path joining $xv$ and $v$.  There are  only finitely many directed  edges in the $G$-orbit
of $e$ in this path.  Hence $g^{-1} \in FH$, where $F$ is finite, and $H = G_e$, and $g \in HF^{-1}$.
Thus $A -Ax^{-1} = HF^{-1}$, i.e. $A- Ax^{-1}$ is $H$-finite.  It follows that both $Ax - A$ and $A-Ax$ are $H$-finite
and so $A + Ax$ is $H$-finite for every $x \in G$,  i.e. $A$ is an  $H$-{\it almost invariant set}.

  If the action on $T$ is non-trivial, then neither $A$ nor $A^*$ is $H$-finite.   We say that $A$ is {\it proper}.
  
  Peter Kropholler has conjectured that the following generalization of Theorem ~\ref{SC} is true for finitely generated groups.
    \begin{conjec}\label{KC}   Let $G$ be a  group and let $H$ be a subgroup.   If there is a proper $H$-almost invariant subset $A$
  such that $A = AH$, then $G$ has a non-trivial action on a tree in which $H$ fixes a vertex $v$ and  every edge  incident with $v$ has an $H$-finite stabilizer.
  
  \end{conjec}

We have seen that  the conjecture   is true if $H$  has one element.    The conjecture has been
proved  for $H$  and $G$ satisfying extra conditions by 
Kropholler \cite {[K90]},  Dunwoody and Roller \cite {[DR]} ,  Niblo \cite {[N]} and  Kar and Niblo \cite {[NK]}.

 If $G$ is the triangle
group $G = \langle a, b | a^2 = b^3 = (ab)^7 =1\rangle$, then $G$ has an infinite cyclic subgroup $H$ for which there 
is a proper $H$-almost invariant set.   Note that in this case $G$ has no non-trivial action on a tree, so the condition $A = AH$ is
necessary in  Conjecture~\ref{KC}.

% in the case when $G$ is finitely generated.  It is hoped to give a proof of the full conjecture in a later paper.

%A discussion of the Kropholler Conjecture is given in \cite{NS} .

\begin{center}
    \includegraphics[width=10cm]{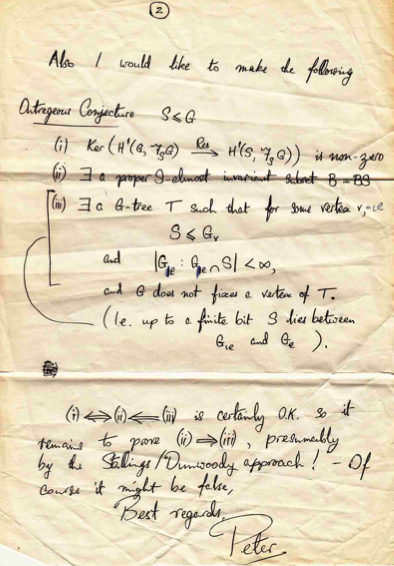}
\end{center}

% in the case when $G$ is finitely generated.  It is hoped to give a proof of the full conjecture in a later paper.

A discussion of the Kropholler Conjecture is given in \cite{NS}.  I first learned of this conjecture in a letter Peter wrote to me in January 1988,  a page of which is shown here.

We give a proof of the conjecture when $G$ is finitely generated over $H$, i.e. it is generated by $H$ together with a finite subset.

I am very grateful to Peter Kropholler for enjoyable discussions and a very helpful email correpondence.

\end{section}

\newcommand{\F}{\mathcal{F}}

\newcommand{\ce}{\mathcal{E}}
\newcommand{\cR}{\mathcal{R}}

\newcommand{\bG}{\bar G}
\newcommand{\bH}{\bar H}
\renewcommand{\d}{{\rm d}}

\newcounter{fig}
\setcounter{fig}{0}
%\DeclareMathOperator{\Cay}{Cay}

%In tis paper it is shown how the theory of finite networks can be generalised to networks based on an arbitrary graphs.

\section {Infinite Networks}
Let $X$ be an arbitrary connected simple graph.  It is not even assumed that $X$ is locally finite.
Let $\B X$ be the set of all edge cuts in $X$.   Thus if $A \subset VX$, then $A \in \B X$ if $\delta A$ is finite.
Here $\delta A$ is the set of edges  which have one vertex in $A$ and one in $A^*$.  

A ray $R$ in $X$ is an infinite sequence $x_1, x_2, \dots $ of distinct vertices such that $x_i, x_{i+1}$ are adjacent for every $i$.
If $A$ is an edge cut, and $R$ is a ray, then there exists an integer $N$ such that for $n > N$ either $x_n \in A$ or $x_n \in A^*$.
We say that $A$ separates rays $R = (x_n), R' = (x_n')$ if for $n$ large enough either $x_n \in A, x_n' \in A^*$ or $x_n \in A^*, x_n' \in A$.
We define $R\sim R'$ if they are not separated by any edge cut.     It is easy to show that $\sim $ is an equivalence relation on the set $\Phi X$ of rays in $X$.   
The set $\Omega X = \Phi X/ \sim $ is the set of {edge} ends of $X$.   An edge cut $A$ separates ends $\omega , \omega '$
if it separates rays representing $\omega , \omega '$.     A cut $A$ separates an end $\omega $ and a vertex $v  \in VX$ if for any ray representing
$\omega $,   $R$ is eventually in $A$ and $v \in A^*$ or vice versa.

%A graph has   more than one (edge) end if there is a cut $A \in \B X$ such that both $A$ and $A^*$ are both infinite.

 We define a network  $N$ to be a simple, connected graph  $X$  and a map $c : EX \rightarrow \{1, 2, \dots \}$.
If  $X$ is  a network
in which each edge has capacity $1$, then $\B X$ is the set of edge cuts, and if $A \in \B X$, then $c(A) = |\delta A|$.

The following result is proved in \cite {[D2]}.

\begin {theo}\label {maintheoremB}  Let $N(X)$ be a network in which $X$ is an arbitrary connected graph.  For each $n >0$, there is a network $N(T_n)$ based on a tree  $T_n$ and a map $\nu : VX\cup \Omega X \rightarrow VT\cup \Omega T$, such
that  $\nu (VX ) \subset VT$ and  $\nu x = \nu y$ for any $x, y \in VX \cup \Omega X$ if and only if $x, y$ are not separated by a cut $A$ with $c(A) \leq n$. 

The network $N(T_n)$ is uniquely determined and is invariant under the automorphism group of $N(X)$.

\end {theo}

Theorem  \ref {maintheoremB} is proved by proving the following lemma.
Let $\B _n X$ be the subring of $\B X$ generated by the cuts $A$ such that $c(A) < n$

\begin {lemma} \label {lem2.2} There is a uniquely defined nested set $\ce _n$ of generators
of $\B _n X$, with the following properties:-
\begin {itemize}
\item [(i)]  If $G$ is the automorphism group of $N(X)$, then $\ce_n $ is invariant under $G$.
\item [(ii)] For each $i <j$, $\ce _i \subseteq  \ce _j$.

\end {itemize}
\end{lemma}

We will only really be using Theorem \ref {maintheoremB} for networks in which every edge has capacity one.
%In this case, the capacity of a cut $A$ is the number of edges in $\delta A$.   The theorem then becomes the following.

\begin {theo}\label {maintheoremC}  Let  $X$ be  a  connected graph.  There is a uniquely determined sequence of structure trees 
  $T_n$ and a map $\nu : VX\cup \Omega X \rightarrow VT\cup \Omega T$, such
that  $\nu (VX ) \subset VT$ and  $\nu x = \nu y$ for any $x, y \in VX \cup \Omega X$ if and only if $x, y$ are not separated by a cut $A$ with $|\delta A|  \leq n$.
Each tree $T_n$ admits an action of the automorphism group of $X$.

\end {theo}

In this case $ET_n = \ce _n$.

In any tree $T$ if $p$ is a vertex and $Q$ is a set of  unoriented edges, then there is a unique set of vertices $P$ such that $v \in P$ then the geodesic $[v,p]$
contains an odd number of edges from $Q$.   We then have $\delta P = Q$.
Note that $\B T = \B_1T$ and every element of $\B T$ is uniquely determined by the set $Q$ together with the information for a fixed  $p\in VT$ whether $p \in A$ or $p \in A^*$.    The vertex $p$ induces an orientation $\co _p$ on the set of pairs $\{ e, \bar e\}  $ of oriented edges by requiring that 
$e \in \co$ if $e$ points at $p$.      For $A \in \B T$,  $A$ is uniquely determined by $\delta A $ together with the orientation $\co _p \cap \delta A$  of the edges of $\delta A$.

In $X$ it is the case that a cut $A$ is uniquely determined by $\delta A$  together with the information for a fixed  $p  \in VX$ whether $p \in A$ or $p \in A^*$.

Since $\B _nX$ is generated by $\ce _n = ET_n$,    the cut $A$ can be expressed in terms of a finite  set of oriented edges of $T_n$.      This set is not usually uniquely
determined.    Thus if $\nu $ is not surjective, and $v$ is not in the image of $\nu $,  and  the set of edges incident with $v$ is finite, then $VX$  is the union of these elements in $\B X$.       The empty set is the intersection of the complements of these sets.     Orienting the edges incident with
$v$ towards $v$ gives the empty set and orienting them away from $v$ gives all of $VX$.
However there is a canonical way of expressing an element of $\B _nX$ in terms of the generating set $\ce _n$.      To see this let $A \in \B _nX  -  \B _{n-1}X$.    There are only finitely many $C \in \ce _n$ with which $C$ is not nested.
This number is $\mu (A, \ce _n)  = \mu (A)$.  We use induction on $\mu (A)$.   Our induction hypothesis is that there is a canonically defined way
of expressing $A$ in terms of the $\ce _n$.    Any two ways of expressing $A$ in terms of $\ce _n$ differ by an expression which gives the empty
set in terms of $\ce _n$.   Such an expression will correspond to a finite set of vertices each of which has finite degree in $T_n$ and  none of which is in the image of $\nu $.    The canonical expression is obtained if there is a unique way of saying whether or not  each such vertex is in the expression for $A$.     Thus the canonical expression for $A$ is determined by  a set of vertices of $VT$ which consists of the vertices of 
$\nu (A)$ together with a recipe for deciding for each vertex which is not in the image of $\nu $ whether it is in the expression for $A$.

Suppose $\mu (A) =0$, so that $A$ is nested with every $C \in \ce _n$, and neither $A$ nor $A ^*$ is empty.      If $A \in \ce _n$, then
this gives an obvious way of expressing $A$ in terms of the $\ce _n$.   If $A$ is not in $\ce _n$, then it corresponds to a unique vertex 
$z \in VT_n$.   Thus because $\mu (A) = 0$,  $A$ induces an orientation of the edges of  $\ce _n$.   To see this,  let $C \in \ce _n$, then  just
one of $C \subset A,  C^* \subset A,  C\subset A^*, C^* \subset A^*$ holds.   From each pair $C, C^*$ we can choose $C$ if $C \subset A$ or
$C \subset A^*$ and we choose $C^*$ if $C^*\subset A$ or $C^* \subset A^*$.    Let $\co$ be this subset of $\ce $.    Then If $C \in \co$ and $D \in \ce
$ and $D \subset C$,  then $D \in \co$.   This means that the orientation $\co$ determines a vertex $z$ in $VT _n$.  Intuitively the edges of $\co $ point at the vertex $z$.
 It can be seen that $A$ or $A^*$ will be the union of finitely many edges $E$ of $\ce _n = ET_n$,  all of which have $\tau E = z$.    If $A$ is such a union, then
we use this to express $A = C_1\cup C_2, \dots \cup C_k$.   If $A$ is not such a union, but $A^* = C_1\cup C_2, \dots \cup C_k$, then we write
$A = (C_1\cup C_2, \dots \cup C_k)^* = C_1^*\cap C_2^*\cap \dots \cap C_k^*$.        
 Note that this gives a unique way of expressing cuts corresponding to a vertex $z$ of finite degree not in the image of $\nu$.   The vertex $z$ is included in the expression for $A^*$ if and only if only  finitely many cuts in $\ce _n$ incident with $z$ and pointing at $z$ are subsets of $A$.
     Suppose then that the hypothesis is true for elements $B \in \B _nX$ for which $\mu (B) < \mu (A)$.
Let $C \in \ce _n$ be not nested with $A$.   Then $ \mu (A\cap C) +\mu (A\cap C^*) \leq  \mu (A)$.      Thus each of $A\cap C$ and $A\cap C^*$
can be expressed in a unique way in terms of the $\ce _n$.  If at most one of these expressions involves $C$ then we take the expression for 
$A$ to be the union of the two expressions for $A\cap C$ and $A\cap C^*$.   If both of the expressions involve $C$, then we take the expression
for $A$ to be the union of the two expression with $C$ deleted.      The expression obtained for $A$ is independent of the choice of $C$.
In fact the decomposition will involve precisely those $C$ for which  $C$ occurs in just one of  the decompositions for  $A\cap C$ and $A\cap C^*$.
We therefore have a canonical decomposition for $A$.     To further  clarify this proof observe the following.   The edges $C$ which are not nested
with $A$ form the edge set of a finite subtree $F$ of $T_n$.  If $EF \not= \emptyset$   we can  choose $C$ so that it is a twig of $F$, i.e. so that one vertex $z$  of $F$
is only incident with a single edge $C$ of $F$.  By relabelling $C$ as $C^*$ if necessary we can assume that $\mu (A\cap C) = 0$. 
The vertex determined by $A\cap C$ as above is $z$,  and we have spelled out the recipe for if this vertex is to be included in the expression
for $A$.   The induction hypothesis gives us a canonical expression for $A\cap C^*$, which together with the expression for $A\cap C$
gives the expression for $A$.

\section {Relative Structure Trees}
We prove Conjecture \ref {KC} in the case when $G$ is finitely generated over $H$, i.e.  $G$ is generated by $H\cup S$ where $S$ is finite.

First, we explain the strategy of the proof.     Suppose that we have a  non-trivial $G$-tree $T$ in which every edge orbit contains an edge  which has an $H$-finite stabiliser, and 
suppose there is a vertex $\bar o$ fixed by $H$.
Let $T_H$ be an   $H$-subtree of  $T$  containing $\bar o$  and every edge with $H$-finite stabiliser.  The action of $H$ on $T_H$ is a trivial action, since 
it has a vertex fixed by $H$, 
and so  the orbit space $H\backslash T_H$ is a tree, which might well be finite, but must have at least one edge.     Our strategy is to show that if
$G$ is finitely generated over $H$ and there is an $H$-almost invariant set $A$ satisfying $AH =A$, then we can find a $G$-tree $T$ with the required properties by first deciding what $H\backslash T_H$ must be and then lifting to get $T_H$ and then $T$.

     We show that if $G$ is finitely generated over $H$, then there is
a $G$-graph $X$ if which there is a vertex with stabiliser $H$ and in which a proper $H$-almost invariant set $A$ satisfying $AH = A$ corresponds to a proper set of vertices with $H$-finite coboundary.     It then follows from the theory of  \cite {[D2]}, described in the previous section, that there is a sequence of structure trees
for $H\backslash X$.      We choose one of these to be $H\backslash T_H$, and show that we can lift this to obtain $T_H$ and then $T$ itself.

For example if $G = {H*}_KL$ then there is a $G$-tree $Y$ with one orbit of edges and  a vertex $\bar o$ fixed by $H$,  and every edge incident with
$\bar o$ has $H$-finite stabiliser.   Suppose that $K, L$ are  such that these are the only edges with $H$-finite stabilisers.   Then $H\backslash T_H$ 
has two vertices and one edge.      When we lift to $T_H$ we obtain an $H$-tree of diameter two in which the middle vertex $\bar o$ has stabiliser
$H$.   The tree $T$ is covered by the translates of $T_H$.   

On the other hand, if $G = L*_KH$ where  $K$ is finite, and $T$ is as above,  then every edge of $T$ is $H$-finite and so $T_H$
is $T$ regarded as an $H$-tree.     The fact that our construction gives a canonical construction for $H\backslash T_H$ means that 
when we lift to $T_H$ and $T$ we will get the unique tree that admits the action of $G$.

We proceed with our proof.
\begin {lemma}  The group $G$ is finitely generated over $H$ if and only if there is a connected $G$-graph $X$ with one orbit of vertices, and finitely
many orbits of edges, and there is a vertex $o$ with stabiliser $H$.
\end {lemma}
\begin {proof}   Suppose $G$ is generated by $H\cup S$, where $S$ is finite.
 Let  $X$ be the graph with $VX = \{ gH | g \in G\}$ and
in which $EX$ is the set of unordered pairs $\{ \{ gH, gsH \},   g \in G, s \in S\} $.    We then have that   $X$ is vertex transitive,  there is a vertex
$o = H$ with stabilizer $H$ and $G\backslash X$ is finite.  We have to show that $X$ is connected.   Let $C$ be the component of $X$ containing 
$o$.     Let $G'$ be the set of those $g \in G$ for which $gH \in C$.   Clearly $G'H = G'$ and $G's = G'$ for every $s \in S$.   Hence $G' = G$ and $C = X$.   Thus $X$ is connected.

Conversely let  $X$ be a connected $G$-graph and $VX = Go$ where $G_o = H$.      Suppose $EX$ has finitely many $G$-orbits,  $Ge_1, Ge_2,
\dots , Ge_r$ where $e_i$  has vertices $o$ and $g_io$.    It is not hard to show that $G$ is generated by $H\cup \{ g_1, g_2, \dots , g_r\} $.

\end{proof}

Let $A \subset G$ be  a proper $H$-almost invariant set satisfying $AH = A$.   Let $G$ be finitely generated over $H$, and let $X$ be a $G$-graph
as in the last lemma.
There is a subset of  $VX$ corresponding to $A$, which is also denoted $A$.      For any $x \in G$,    $A +Ax$ is $H$-finite.   In particular
this is true if $s \in S$.   This means that $\delta A$ is $H$-finite.  Note that neither $A$ nor $A^* = VX - A$ is $H$-finite.
Thus a proper $H$-almost invariant set corresponds to a proper subset of $VX$ such that $\delta A$ is $H$-finite.

 From the previous section (Lemma \ref{lem2.2}) we know that $\B(H\backslash X)$ has a uniquely determined nested set of generators $\ce = \ce (H\backslash X)$.
 For $E \in \ce $, let $\bar E \subset VX$ be the set of all $v \in VX$ such that $Hv \in E$.   Let $C$ be a component of $\bar E$.  
 \begin {lemma}  For $h \in H$,  $hC = C$ or $hC\cap C = \emptyset$.   Also $HC = \bar E$,  $h\delta C \cap \delta C =\delta C$ or $h\delta C \cap \delta C = \emptyset$ and $H\backslash \delta C = \delta E$.   
 \end {lemma}
 \begin {proof} Let $h \in H$.   Then $hC$ is also a component of $\bar E$, since $HC \subseteq E$.
Thus either $hC =C$ or $hC\cap C = \emptyset $.   Let $K$ be the stabilizer of $C$ in $H$.    if $v  \in  C$ then $hv \in  C$ if and only if
$h \in K$. Thus $K\backslash C$ injects into $H\backslash C = E$  and $K \backslash \delta C$ injects into $\delta E$.   But $E$ is connected,
and so the image $HC$ is $E$.  It follows that there is a single $H$-orbit of components.

  \end {proof}
 It follows from the lemma  that it is also   the case that $C^*$ is connected, since any component of $C^*$ must have coboundary 
 that includes an edge from each orbit of $\delta C$.    Let $\bar \ce (H, X)$ be the set of all such $C$, and let $\bar \ce _n(H, X)$  be the subset of $\bar \ce (H, X)$ corresponding to those $C$ for which $\delta C$ lies in at most $n$  $H$-orbits.
 
 \begin {lemma}   the set $\bar \ce (H,X)$ is a nested set.  The set $\bar \ce _n(H, X)$ is the edge set of an $H$-tree.
 \end {lemma}
 \begin {proof}  Let $C, D \in \bar  \ce _n(H,X)$.    Then $HC, HD$ are in the nested set $\ce $.    Suppose $HC \subset HD$,   then $C \subset D$ or
 $C\cap D = \emptyset $.  It follows easily that $\bar \ce (H, X)$ is nested.   It was shown in \cite {[D1]} that a nested set $\ce $ is the directed edge set of
 a tree if and only if it satisfies the finite interval condition, i.e. if $C, D \in \ce $ and $C \subset D$,  then there are only finitely many $E \in \ce $ such that
 $C \subset E \subset D$.
 Thus we have to show that $\bar \ce _n(H, X)$ satisfies the finite interval condition.  If  $C\subset D$ and $C\subseteq E \subseteq D$ where
 $C, E,D \in \bar \ce _n(H,X)$, then $HC \subseteq HE \subseteq HD$.  But $\ce _n(H, X)$ does  satisfy the finite interval condition and 
 $HC = HE$ implies $C = E$.     Now let $C\cap D = \emptyset $ and suppose that $o = H \in C^*\cap D^*$.  There are only finitely
 many $E \in \bar \ce _n$ such that $C \subset E$ and $o \in E^*$ or such that $D \subset E^*$ and $o \in E$.     Each $E \in \bar \ce _n$ 
 such that $C \subset E \subset D^*$ has one of these two properties.
 
 \end {proof}

Let $\bar T = \bar T(H)$ be the tree constructed in the last Lemma.   Let $T  = H\backslash \bar T$.
Note that in the above $\bar T(H)$ is the Bass-Serre   $H$-tree associated with the quotient graph $T(H) = H\backslash \bar T(H)$ and the graph
of groups obtained by associating appropriate labels to the edges and vertices of this quotient graph (which is a tree).    Clearly the action of $H$ on $T(H)$ is a trivial action in that $H$ fixes the vertex $\bar o = \nu o$.   The stabilisers of edges or vertices on a path or ray beginning at $\bar o$ will form a non-increasing sequence of subgroups of $H$.
 
We now adapt the argument of the previous section to show that if $A \subset VX$ is such that $\delta A$ lies in at most $n$  $H$-orbits,  then there is a canonical way
of expressing $A$ in terms of the set $\bar \ce  (H, X)$.    In this case we have to allow unions of infinitely many elements  of the generating set.
Our induction hypothesis is that if $\delta A$ lies in at most $n$ $H$-orbits, then $A$ is canonically  expressed in terms of $\bar \ce _n (H, X)$.
First note that there are only finitely many $H$-orbits of elements of $\bar \ce _n = \bar \ce _n(H,X)$ with which $A$ is not nested.  This is because 
if $C \in \bar \ce _n$ is not nested with $A$ and $F$ is a finite connected subgraph of $H\backslash X$ containing all the edges of $H\delta A$,
then $H\delta C$ must contain an edge of $F$ and there are only finitely many elements of $\ce _n$ with this property.
We now let $\mu (A)$ be the number of $H$-orbits of elements of $\bar \ce _n$ with which $A$ is not nested.
If $\mu (A) = 0$, then $A$ is nested with every $C \in \bar \ce _n$.   This then means that if neither $A$ nor $A^*$ is empty and it is not
already in $\bar\ce _n$, then $A$ determines a vertex $z$  of $\bar T_n$  and either $A$ or $A^*$ is the union (possibly infinite) of edges of
$T_n$ that lie in finitely many $H$-orbits.  
 If $A$ is such a union, then
we use this union for our canonical expression for $A$.   If $A$ is not such a union, then $A^*$ is;   we have $A^* = \bigcup \{ C_{\lambda} | \lambda \in \Lambda \}$, where each $C_{\lambda }$ has $\tau C_{\lambda } = z$ and the edges lie in finitely many $H$-orbits.   We write $A = ( \bigcup \{ C_{\lambda} | \lambda \in \Lambda \})^*  = \bigcap    \{ C_{\lambda}^* | \lambda \in \Lambda \}$.         Note that this gives a canonical way of expressing cuts corresponding to a vertex 
that is not in the image of $\nu $ and whose incident edges lie in finitely many $H$-orbits.        Suppose then that the hypothesis is true for elements $B$ for which $\mu (B) < \mu (A)$.
Let $C \in \bar  \ce _n$ be not nested with $A$.   Then $ \mu (A\cap HC) +\mu (A\cap HC^*) \leq  \mu (A)$.      Thus each of $A\cap HC$ and $A\cap HC^*$
can be expressed in a unique way in terms of the $\ce _n$.   We take the expression for 
$A$ to be the union of the two expressions for $A\cap HC$ and $A\cap HC^*$ except that we include  $hC$  for $h \in H$, only if just one of the two expressions involve $hC$.

If $g \in G$, then $g\bar T(H)$ is a $(gHg^{-1})$-tree.  It is the tree $\bar T(gHg^{-1})$ obtained from the $G$-graph $X$ by using the vertex $go$
instead of $o$.    We now show that there is a $G$-tree $T$ which contains all of the trees $g\bar T(H)$.

We know that the action of the group $G$ on $X$  is vertex transitive and that $X$ has a vertex $o$ fixed by $H$.  Also $G$ is generated by
$H\cup S$ where $S$ is finite.

Clearly there is an isomorphism  $\alpha _g : \bar T(H) \rightarrow \bar T(gHg^{-1})$ in which $D \mapsto gD$.

Suppose now that $\nu o \not= \nu (go)$. 
Let  $A, B $ be $H$-almost invariant sets satisfying $AH = A, BH = B$ and let $g \in G$.  We regard $A,B$ as subsets of $VX$, so that 
$\delta A$ and $\delta B$ are $H$-finite.

Suppose that $o \in gB^*$ and $go\in A^*$. 
 The following Lemma is due to Kropholler \cite {[K90]}, \cite {[K91]}.
 We put $K = gHg^{-1}$.
\begin {lemma}\label {Krop}  In this situation $\delta (A\cap  gB)$is  $(H\cap K)$-finite.
\end {lemma}
\begin {proof}   Let $x \in G$.  We show that the symmetric difference $(A\cap gB)x + (A\cap gB)$ is $(H\cap K)$-finite.
Since $A, B$ are $H$-almost invariant, there are finite sets $E, F$ such that $A+ Ax \subseteq HE$ and $B+Bx\subseteq HF$.
We then have

  $$ 
   (A\cap gB)x + (A\cap gB) = Ax\cap (gBx +gB) + (Ax + A)\cap gB 
= Ax\cap gHF + g(g^{-1}HE\cap B). 
  $$
  
  Now $Ax\cap gHF$ is $K$-finite, but it is also $H$-finite because $gH$ is contained in $A^*$, since $go \in A^*$. A set which is both $H$-finite and $K$-finite is $H\cap K$-finite.  Thus $Ax\cap gHF$ is $(H\cap K)$-finite.  Similarly using the fact that  $g^{-1}o \in B^*$, it follows that $g^{-1}HE\cap B$ is $H\cap (g^{-1}Hg)$-finite, and so $g(g^{-1}HE\cap B)$ is $(H\cap K)$-finite.  Thus $A\cap gB$ is $(H\cap K)$-almost invariant.  
  But this means that $\delta (A\cap gB)$ is $(H\cap K)$-finite.
  \end {proof}  
  What this Lemma says is that if $A, gB$ are not nested then there is a special corner - sometimes called the {\it Kropholler  corner} - which is $(H\cap K)$-almost invariant.    
    
 Notice that in the above situation all of $\delta A, \delta (A\cap gB^*)$ and $\delta  (A\cap gB)$ are $H$-finite.
 If we take the canonical decomposition for $A$, then it can be obtained from the canonical decompositions for $A\cap gB$ and $A\cap gB^*$
 by taking their union and deleting any edge that lies in both.   Also $\delta (gB)$ is $K$-finite and the decomposition for
 $gB$ can be obtained from those for $gB\cap A$ and $gB\cap A^*$.     But the edges in the decomposition for $A\cap gB$ which
 is $(H\cap K)$-almost invariant are the same in both decompositions.

  We will now show that it follows from Lemma \ref {Krop} that the set $G\bar \ce _n$ is a nested $G$-set which satisfies the final interval condition,
  and so it is the edge set of a $G$-tree.
  We have seen that $\bar \ce _n$ is a nested $H$-set where  $\ce _n = H\backslash \bar \ce _n$ is the uniquely determined  nested subset of $\B _n(H\backslash X)$ that generates
  $\B _n (H\backslash X)$ as an abelian group. It is the edge set of a tree $T _n(H\backslash X)$.
  
  %\eject
\begin{figure}[ht!]

\centering
\begin{tikzpicture}[scale=.6]

    \draw (0,4)--(8,4);
    \draw (4,0)--(4,8);

\draw (3.6, 7.8) node  {$A$};
\draw (4.6,7.8) node  {$A^*$};
\draw (.4, 4.4) node  {$gB$} ;
\draw (.4,3.6) node  {$gB^*$} ;
\draw (6,6) node  {$A^*\cap gB$};
\draw (2,6) node  {$A\cap gB$};
\draw (2,2) node  {$A\cap gB^*$};
\draw (6,2) node  {$A^*\cap gB^*$};

    \draw (10,5)--(13,5)-- (13, 8);
       \draw (10,3)--(13,3)-- (13, 0);
\draw (18,5)--(15,5)-- (15, 8);
       \draw (18,3)--(15,3)-- (15, 0);

   % \draw (14,0)--(14,8);
\draw [dashed] (13,3)--(15,5) ;
\draw [dashed] (15,3)--(13,5) ;

\draw (17,6) node  {$A^*\cap gB$};
\draw (11,6) node  {$A\cap gB$};
\draw (11,2) node  {$A\cap gB^*$};
\draw (17,2) node  {$A^*\cap gB^*$};

\draw (14,7) node  {$a$};
\draw (14,1) node  {$b$};
\draw (11, 4) node  {$c$};
\draw (17, 4) node  {$d$};
\draw (13.5, 3.5) node  {$e$};
\draw (14.5, 3.5) node  {$f$};

  \end{tikzpicture}

\vskip .5cm \caption{Crossing cuts}\label{Cuts}
\end{figure}
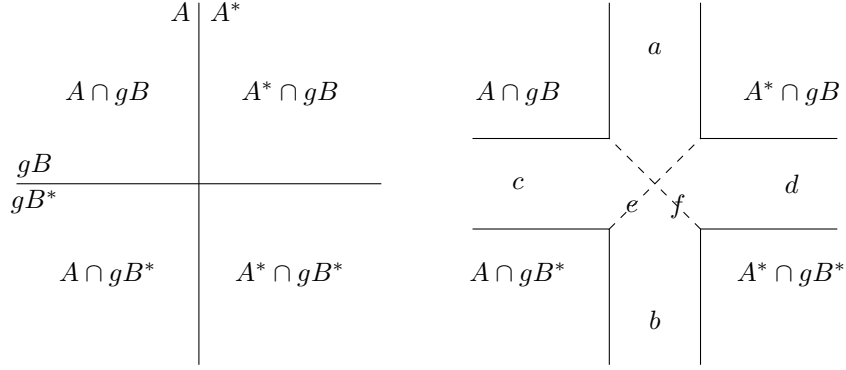

  If $A, B \in  \bar \ce _n$ and $A, gB$ are not nested for some $g \in G$, then by Lemma \ref {Krop} there is a corner -the Kropholler corner -, which we take to be $A\cap gB$, for which $\delta (A\cap gB)$ is $(H\cap K)$-finite.    We than have canonical decompositions for $A\cap gB$ and $A\cap gB^*$ as above.
  This is illustrated in Fig \ref {Cuts}.     The labels  $a, b, c, d, e, f$ are for sets of edges joining the indicated corners.  In this case the letters do not represent edges of $X$ but elements of $\bar \ce _n$.    Although each $E \in \bar \ce _n$ comes with a natural direction, in the diagram we only count 
  the unoriented edges,   i.e. we count the number of edge pairs $(E, E^*)$.    In the diagram,  $A\cap gB$ is always taken to be the Kropholler corner.
  Thus we have that any pair contributing to $a, f$ or $e$ must be $(H\cap K)$-finite.   Any pair contributing to $e$ or $b$ must be $H$-finite and any
  pair contributing to $e$ or $d$ must be $K$-finite.
  
  We have that $a + e + f +b = 1$ and $c + e + f + d =1$.     Suppose that the Kropholler corner $A\cap B$ is not empty. It is the case that  each of $o$
  and $go$ lies in one of the other three corners.     We know that $o \in gB^*, go \in A^*$.    If $o \in A\cap gB^*$ and $go \in A^*\cap gB$, then 
  $a = c = 1$ and $e =f = b = d =0$ and $A^*\cap gB^* = \emptyset$.      If $o \in A^*\cap gB$ and $go \in A^*\cap gB^*$, then $a = d =1$ and 
  $A\cap gB^* = \emptyset$,  while if both $o$ and $go$ are in $A^*\cap gB^*$,  then either $a =d =1$  and $A\cap gB^* = \emptyset$ or
   $a =c =1$  and $A^*\cap gB = \emptyset$ or  $f =1$ and both
  $A\cap gB^*$ and $A^*\cap gB$ are empty, so that  $A = gB$.  In all cases $A, gB$ are nested.

\begin{figure}[htbp]
\centering
\begin{tikzpicture}[scale=.5]
\draw [red]  (-.8, 1) -- (-.5,2) -- (-.5,3)-- (-1, 4) -- (-1, 5) -- (-1.4, 4) -- (-1. 6, 3)  ;
\draw [thick, dashed, red]   (-2, 2) --(-1.6,3) -- (-1.4, 2) ;
\draw [thick, brown]   (-.5,0) -- ( -.3, 1) --(-.5, 2)-- (1,2)--(2,1)--(3,2)--(4,1)--(5,2)--(6,1)--(7,2)--(8,1)--(9,2)--(10.5,2)  ;
\draw (-1,5) node {$\bullet$} ;
\draw (10.5,5) node {$\bullet$} ;
\draw [thick, blue]  (10.5,2) --(10.5, 3) --(11,4)--(10.5, 5) -- (11.5, 4) --( 11.5, 3)  ;
\draw [thick, brown] (3,-1) --( 2.5, 0) -- (3, 2) -- (3.5, 0) ;
\draw [thick, brown] (4.5,-1) --( 4.5, 0) -- (5, 2) -- (5.5, 0) ;

\draw (-1,5) node [above] {$o$} ;
\draw (10.5,5) node [above] {$go$} ;
\end{tikzpicture}
\end{figure}

 We need also to show that $G\bar \ce _n$ satisfies the finite interval condition.    Let $g \in G$ and let $K = gHg^{-1}$.  Consider the union
 $\bar \ce \cup  g\bar \ce $.     This will be a nested set.    In fact it will be the edge set of a tree that is the union of the trees $T(H)$ and $T(K)$.
 In the diagram the red edges are the edges that are just in $T(H)$.   The blue edges are the ones that are in $T(K)$.  The brown edges are in
 both $T(H)$ and $T(K)$.   An edge is in the geodesic joining $o$ and $go$ if and only if it has stabiliser containing $H\cap K$, it will also
 lie in both $T(H)$ and $T(K)$ (i.e.  it is coloured brown) if and only if it its stabiliser contains $H\cap K$ as a subgroup of finite index.
 It may be the case that $T(H)$ and $T(K)$ have no edges in common, i.e. there are no brown edges.    
 An edge lies in both trees if and only if it has a stabiliser that is $(H\cap K)$-finite.   It there are such edges then they will be the edge set of
 a subtree of both trees. They will correspond to the edge set  $\bar \ce (H\cap K)$.

It  follows that $T(H)$ is always a subtree of a  tree constructed from a subset of $G\bar \ce _n$ that contains $\bar \ce _n$.
If $T(H)$ and $T(K)$ do have an edge in common, then $T(H)\cup T(K)$ will be a subtree of the tree we are constructing.   
If $e \in EX$ has vertices $go$ and $ko$ and there is some $C \in G\bar \ce _n$ that has $e \in \delta C$, then $C \in gET(g^{-1}Hg)\cap kET(k^{-1}Hk)$.   If there is no such $C$,  i.e. there is no cut $C \in G\bar \ce _n$ that separates $o$ and $k^{-1}go$ then $T(H) = k^{-1}gT(H)$.
As there is a finite path connecting any two vertices $u,v$  in $X$, it can be seen that there are only finitely many edges in $G\bar \ce _n$ separating
$u$ and $v$ since any such edge must separate the vertices of one of the edges in the path.   Thus $G\ce _n$ is the edge set of a tree.

%    \end {proof}

  We say that a $G$-tree $T$ is reduced if for every $e \in ET$, with vertices $\iota e$ and $\tau e$  we have that either  $\iota e$  and $\tau e$ are in the same orbit, or $G_e$ is a proper subgroup of both $G _{\iota e} $ and $G _{\tau e} $.
  
   \begin {theo}\label {tree}
   Let $G$ be a group that is finitely generated over a subgroup $H$.
   The following are equivalent:-
    \begin {itemize}
    \item [(i)]
  There is a proper $H$-almost invariant set $A = HAK$ with left stabiliser $H$ and right stabiliser $K$, such that $A$ and $gA$ are nested for every $g \in G$.
  \item [(ii)]  There is a reduced $G$-tree $T$ with vertex $v$ and incident edge $e$ such that $G_v = K$ and $G_e = H$.   
  \end {itemize}
  \end {theo}
  \begin {proof}
  It is shown that (ii) implies (i) in the Introduction.
  
  Suppose than that we have (i).     We will show that there is a $G$-tree - in which $G$ acts on the right - which contains the set 
  $V =  \{ Ax | x \in G\}$ as a subset.        Let  $x \in G$,   then $A + Ax$ is a union of finitely many coset $\{ Hg_1, Hg_2, \dots , Hg _k$.
  Then $\{ g_1 ^{-1}A, g_2^{-1}A, \dots , g_k^{-1}A  \}$ is the edge set of a finite tree  $F$.
  We know that the set $\{ gA | g \in G\}$ is the edge set of a $G$-tree $T$ provided we can show that it satisfies the finite interval condition.
 But  this must be the case as the edges separating vertices $A$ and $Ax$ will be the edges of  $F$.  
    
  \end {proof}
    \begin {theo} Let $G$ be a  group and let $H$ be a subgroup, and suppose $G$ is finitely generated over $H$.   There is a proper $H$-almost invariant subset $A$
  such that $A = AH$, if and only if  there is a non-trivial  reduced $G$-tree $T$ in which $H$ fixes a vertex and every edge orbit contains an edge  with an $H$-finite edge stabilizer.
  \end {theo}
  
   \begin {proof}
   The only if part of the theorem is proved in Theorem \ref {tree}.   In fact it is shown there that if $G$ has an action on a tree with the specified properties, then there is a proper $H$ almost invariant set $A$  for which $HAH= A$.
   
   Suppose then that $G$ has an $H$-almost invariant set $A$ such that  $AH = A$.  Since $G$ is finitely generated over $H$, we can construct
   the $G$-graph $X$ as above, in which $A$ can be regarded as a set of vertices for which $\delta A$ lies in finitely many $H$-orbits.
   Let this number of orbits be $n$.    Then we have seen that there is a $G$-tree $\bar T_n$ for which $H$ fixes a vertex $\bar o$ and
   every edge is in the same $G$-orbit as an edge in $\bar T(H)$.   The edges in this tree are $H$-finite.  The set $A$ has an expression in terms 
   if the edges of $\bar T(H)$.
   Finally we need to show  that the action on $\bar T_n$ is non-trivial.      If $G$ fixes $\bar o$, then $\nu (A)$ consists of the single vertex
   $o$ and so $A$ is not proper.     In fact the fact that $A$ is proper ensures that no vertex of $\bar T_n$ is fixed by $G$.

   It can be seen from the above that $ \bar T(H) \cap \bar T(g^{-1}Hg)  = \bar T(H\cap gHg^{-1})$  so that  if $e \in ET(H)$, and $g \in G_e$, then $e\in \bar T(gHg^{-1})$ and so  $G_e$ is $H$-finite.
  \end {proof}

    The Kropholler Conjecture follows immediately from the last Theorem.

 \section {$H$-almost stability}
Let $G$ be a group with subgroup $H$, and let $T$ be a $G$-tree.

Let $ \bar  A \subset VT$ be  such that $\delta \bar A \subset ET$  consists of finitely many $H$-orbits of edges $e$ such that $G_e$ is $H$-finite.  Also let $H$ fix a vertex of $T$.    Note that $\delta \bar A$ consists of whole $H$-orbits, so that $e \in \delta \bar A$ implies
$he \in \delta \bar A$ for every $h \in H$.     The fact that $G_e$ is $H$-finite for $e \in \delta \bar A$  follows from the fact that $\delta \bar A$ is $H$-finite.      If $H_e$ is the stabiliser of $e \in \delta \bar A$, then $[G_e : H_e]$ is finite.

%If $v = \iota e $, then $G_e = H \leq K = G_v$ and if $A = G[e, \iota e]$, then $A= HAK$.

Let $v \in VT$,  and   let $A = A(v)  = \{ g \in G | gv \in \bar A\} $.      Note that $A(xv) = A(v)x^{-1}$, so that the left action on $T$ becomes a right 
action on the sets $A(v)$.
 if $x \in G$ and $[v , xv]$ is the geodesic from $v$ to $xv$, then $g \in A + Ax$ if and only the geodesic $[gv, gxv]$ contains an odd number of
 edges in $\delta \bar A$.   If $[v, xv]$ consists of the edges $e_1, e_2, \dots , e_r$,  then $ge_i \in \delta \bar A $ if and only if $Hge_i \in \delta \bar A$.   It follows that $H(A + Ax) = A+Ax$.   It is also clear that for each $e_i$ there are only finitely many cosets $Hg$ such that $Hge_i \in \delta \bar A$.
 Thus $A$ is $H$-almost invariant.     We also have $A (v)H = A(v)$ if $H$ fixes $v$.
 
 For each $e \in ET$, let $d(e)$ be the number of cosets $Hg$ such that $Hge \in \delta \bar A$.    We see that $d(e) = d(xe)$ for every $x \in G$
 and so we have a metric on $VT$, that is invariant under the action of $G$.     We will show that if $G$ has an $H$-almost invariant set such
 that $HAH = A$ then there is a $G$-tree with a metric corresponding to this set.

From now on we are interested in the action of $G$ on the set of $H$-almost invariant sets.   But note that we are interested in
the action by right multiplication.   The  Almost Stability Theorem \cite {[DD]}, also used 
the action by right multiplication. 
Let $A \subset G$ be $H$-almost invariant and let $HA = A$     For the moment we do not assume that $AH = A$.

Let $M  = \{ B | B =_a A\}  $ so that  for  $B, C   \in M ,     B +C  = HF$ where $F$ is finite.

Note  that for $H = \{ 1 \}$  it follows from the Almost Stability Theorem that $M$ is the vertex set of a $G$-tree.

%We now complete our proof  of Theorem\ref{KC} for general $H$, by using  the following structure on $X$ as a metric space.
We define a metric on $M$.
For $B, C \in M$ define $d(B, C)$ to be the number of $H$-cosets  in $B +C$.

This is a metric on $M$,   since  $(B+C) + (C+D) = (B+D)$, and so an element  which is in $B+D$ is in just one of 
$B+C$ or $C+D$.  Thus $d(B,D) \leq d(B,C) + d(C, D)$.

Also $G$ acts on $M$ by right multiplication and this action is  by isometries, since $(B+C)z = Bz + Cz$. 
Let $ \Gamma $ be the graph with $V \Gamma = M$ and two vertices are joined by an edge if they are distance one apart.
   Every edge in $ \Gamma $ corresponds to a particular $H$-coset.
There are exactly  $n!$ geodesics joining $B$ and $C$ if $d(B, C) = n$, since a geodesic will correspond to a permutation of
the cosets in $B + C$.  The vertices of $ \Gamma$ on such a geodesic form the vertices of an $n$-cube.

 The edges corresponding to a particular coset $Hb$ disconnect $\Gamma $, since removing this set of edges gives two sets of vertices,   $B$ and $B^*$, where $B$ is the set of those $C \in M$ such that $Hb \subset C$.

It has been pointed out to me by Graham Niblo that $\Gamma $ is the $1$-skeleton of the Sageev cubing introduced in \cite {[Sa]}.
For completeness we describe this alternative characterization of $\Gamma $.

 Let $G$ be a group with subgroup $H$ and let $A = HA$ be an $H$-almost invariant subset.  Let 
 
 $$ \Sigma = \{ gA | g \in G\} \cup \{gA^* | g \in G \}.$$
 
We define a graph $\Gamma '$.
 A vertex $V$ of $\Gamma '$ is a subset of  $\Sigma $ satisfying the following conditions:-
 \begin {itemize}
 \item [(1)] For all $B \in \Gamma '$, exactly one of $B, B^*$ is in $V$.
 \item[(2)] If $B \in V, C \in \Sigma $ and $B \subseteq C$, then $C \in V$.
 \end{itemize}
Two vertices are joined by an edge in $\Gamma '$ if they differ by one element of $\Sigma $.
For $g \in G$, there is a vertex $V_g$ consisting of all the elements of $\Sigma $ that contain $g$.
Then Sageev shows that there is a component $\Gamma ^1$ of $\Gamma '$ that contains all the $V_g$.   In fact this graph
$\Gamma ^1$ is isomorphic to our $\Gamma $.

By (1) for each $V \in \Sigma $ either $A \in V$ or $A^* \in V$ but not both.  Let $\Sigma _A$ be the  subset of $\Sigma $ consisting
of those $V \in \Sigma $ for which $A \subset V$.    The edges joining $\Sigma _A$ and $\Sigma _A^*$ in $\Gamma ^1$ form a hyperplane.
Each edge in the hyperplane joins a pair of vertices that differ only on the set $A$.   For each $xA$ there is a hyperplane joining
vertices that differ only on the set $xA$.  Clearly $G$ acts transitively on the set of hyperplanes.

With $V$ as above,  consider the subset $A_V$ of $G$ 
   $$ A_V =\{ x \in G | x^{-1}A \in V \} . $$
   Then $HA_V = A_V$ and $A_{V_1} =A$.
   Also $A_V +A$ is the union of those cosets $Hx$ for which $V$ and $V_1$ differ on $x^{-1}A$, which is finite.   Thus $A_V \in V\Gamma $.

Thus there is a map $V\Gamma ^1 \rightarrow V\Gamma  $ in which $V \mapsto A_V$.
This map is a $G$-map and an isomorphism of graphs.

If the set $A$ is such that $A$ and $gA$ are nested for every $g \in G$, then  there is a $G$-subgraph of $\Gamma _1$ which is a $G$-tree.
This will also be true of $\Gamma $.

In $\Gamma $ a hyperplane consists of edges joining those vertices that differ only by a particular coset $Hx$.   Every edge of $\Gamma $
belongs to just one hyperplane.   The group $G$ acts transitively on hyperplanes.   The hyperplane corresponding to $Hx$ has stabilizer
$x^{-1}Hx$.

Suppose now that   $A$ is $H$-almost invariant with $HAK = A$.     Here $H$ is the left stabiliser and $K$ is the right stabiliser of $A$, and we assume that  $H \leq K$, so that in particular $HAH = A$.
Note that it follows from the fact that $A$ is $H$-almost invariant that it is also $K$ almost invariant.
Suppose that $G$ is finitely generated over $K$.     We have seen, in the previous section,  that there is a $G$-tree $T$ in which $A$ uniquely determines a set  $\bar A$ of vertices with
$H$-finite coboundary $\delta \bar A$.  Here $T = T_n$ for $n$ sufficiently large that in the graph $X$ -as defined in the previous section -
the set $\delta \bar A $ is contained in at most $n$ $H$-orbits of edges.
Note that if $e$ is an edge of $\bar T (H) = \bar \ce  (H, X)$, then $\delta e$ is $H_e$-finite, and will consist of finitely many $H_e$-orbits.   It is then 
the case that $[G_e: H_e]$ is finite, since $\delta e$ will consist of finitely many $G_e$-orbits each of which is a union of $[G_e:H_e]$  $H_e$-orbits
of edges.

We also know that $K$ fixes a vertex $\bar o$ of $T$,  and that $H\delta \bar A = \delta \bar A$.  
Thus $\delta \bar A$ consists of finitely many $H$-orbits of edges.   
  We can  contract any edge whose $G$-orbit does not intersect $\delta \bar A$.
We will then have a tree that has the properties indicated at the beginning of this section.   Thus $ \bar  A \subset VT$ is   such that $\delta \bar A \subset ET$  consists of finitely many $H$-orbits of edges $e$ such that  $G_e$ is $H$-finite.
We see that the metric $d$ on $M$   is the same as the metric defined on $VT$.  Explicitly we have proved the following theorem in the case when
$G$ is finitely generated over $K$.
\begin {theo}
Let $G$ be a group with subgroup $H$ and let $A = HAK$ where $H \leq K$ and $A$ is $H$-almost invariant.  
Let $M$ be the $G$-metric space defined above.    Then there is a $G$-tree $T$ such that  $VT$ is a $G$-subset of $M$  and the metric on
$M$ restricts  to a geodesic metric on  $VT$.  If $e \in ET$ then some edge in the $G$-orbit of $e$ has $H$-finite stabiliser.
\end {theo}

This is illustrated in Fig 1 and Fig 2.

\begin {proof}
It remains to sow that that the theorem for arbitrary $G$ follows from the case when $G$ is finitely generated over $K$.
Thus if $F$ is a finite subset of $G$,  then there is a finite convex subgraph $C$  of $\Gamma $ containing $AF$.   We can use the graph $X$ of the previous section for the subgroup $L$ of $G$ generated by $H \cup F$ to construct an $L$-tree which has a subtree  $S(F)$ with vertex set contained
in $VC$.  These subtrees have the nice property that if $F_1 \subset F_2$ then $S(F_1)$ is a subtree of $S(F_2)$.    They therefore fit together
nicely to give the required $G$-tree.      We give a more detailed argument for why this is the case.  We follow the approach of \cite {[C]}.

Let $M'$ be the subspace of $M$ consisting of the single $G$-orbit $AG$.    Define an inner product on $M'$  by 
$(B.C)_A = {1\over 2} (d(A, B) +d(A,C) - d(B, C)).$

This turns $M'$ into a $0$-hyperbolic space, i.e. it satisfies the inequality 
$$(B.C)_A \geq min \{ (B.D)_A, (C.D)_A\}$$
for every $B, C, D \in M'$.    This is because we know that if $L \leq G$ is finitely generated over $H$, then there is an $L$-tree
which is a subspace of $M$.   But $A, B, C,D$ are vertices of such a subtree which is $0$-hyperbolic.
It now follows from \cite {[C]},  Chapter 2, Theorem 4.4 that there is a unique $\Z$-tree $VT$ (up to isometry) containing $M'$.    The subset of   $VT$ consisting of vertices of degree larger than  $2$  will be the vertices of a $G$-tree and can be regarded as a $G$-subset of $M$ containing $M'$.
\end {proof}

%\eject

\begin{figure}[htbp]
\centering
\begin{tikzpicture}[scale=1]
 \path (6,2) coordinate (p1);
    \path (2,3) coordinate (p2);
    \path (5,4) coordinate (p3);
   \draw (4, 3) coordinate (q1);
   \draw (3, 5) coordinate (q3);
    \draw (5.5, 5) coordinate (q4);
    \draw (1.5, 2) coordinate (q5);
    \draw (2.5, 1) coordinate (q6);
     \filldraw (p1) circle (3pt) ;
 \filldraw (p2) circle (3pt);
   \filldraw (p3) circle (3pt);  
\filldraw  (q1) circle (3pt);
\filldraw [white]  (q1) circle (2pt);

\path (3.5, 6) coordinate (p4);
 \filldraw (p4) circle (3pt);  
 \path (3, 2) coordinate (p5);
  \filldraw (p5) circle (3pt);  
% \draw (p1) -- (q1);
%  \draw (p2) -- (q1);
%\draw (p3) -- (q1);
 \path (7.5, 2) coordinate (q2);
 \filldraw (q2) circle (3pt);
\filldraw (q3) circle (3pt);
\filldraw (q4) circle (3pt);
\filldraw (q5) circle (3pt);
\filldraw (q6) circle (3pt);

\draw [red] (q3)--(p4) ;
\draw [red] (q4)--(p3) ; 
\draw [red] (p2)--(q5)--(p5);  
 %\draw (p1) --(q2);
 \draw [red] (4,3) --(5.1,3) --(6, 2.7) -- (6,2) ;
\draw [red] (4,3) --(4.9,2.7) --(6, 2.7) ;
\draw [red] (4,3) --(4,2.3) --(4.9, 2) -- (6,2) ;
\draw [red] (4,2.3) --(5.1, 2.3) -- (6,2) ;
\draw [red] (5.1, 2.3) -- (5.1,3) ;
\draw [red] (4.9, 2.7) -- (4.9,2) ;
\draw [red] (4, 3)--(4.5, 3.2) --(5, 4) -- (4.5, 3.8) -- cycle ;
\draw [red] (4, 3)--(3, 3.3) --(2,3) -- (3,2.7) -- cycle ;

 \draw [red] (3.5,6) --(4.6,6) --(5.5, 5.7) -- (5.5,5) ;
\draw [red] (3.5,6) --(4.4,5.7) --(5.5,5.7) ;
\draw [red] (3.5,6) --(3.5,5.3) --(4.4, 5) -- (5.5,5) ;
\draw [red] (3.5,5.3) --(4.6,5.3) -- (5.5,5) ;
\draw [red] (4.6, 5.3) -- (4.6,6) ;
\draw [red] (4.4, 5.7) -- (4.4,5) ;
\draw [red] (6,2)--(6.75, 2.5) --(7.5,2) -- (6.75, 1.5) -- cycle ;
\draw [red] (2.5,1)--(2.5, 1.5) --(3,2) -- (3,1.5) -- cycle ;

    \filldraw [white] (p1) circle (2pt);
\filldraw [white]  (p2) circle (2pt);
\filldraw [white] (p3) circle (2pt);
\filldraw [white] (p5) circle (2pt);
\filldraw [white] (p4) circle (2pt);
\filldraw [white]  (q1) circle (2pt);
 \filldraw [white] (q2) circle (2pt);
\filldraw  [white] (q3) circle (2pt);
\filldraw [white]  (q4) circle (2pt);
\filldraw [white]  (q5) circle (2pt);
\filldraw [white] (q6) circle (2pt); 

%{white  \filldraw (q2) circle (2pt);

\end{tikzpicture}
%\vskip-4.8cm \caption{One-connected graph and structure tree}\label{fig:1block}
\end{figure}

\begin{figure}[htbp]
\centering
\begin{tikzpicture}[scale=1.2]
 \path (6,2) coordinate (p1);
    \path (2,3) coordinate (p2);
    \path (5,4) coordinate (p3);
   \draw (4, 3) coordinate (q1);
   \draw (3, 5) coordinate (q3);
    \draw (5.5, 5) coordinate (q4);
    \draw (1.5, 2) coordinate (q5);
    \draw (2.5, 1) coordinate (q6);
     \filldraw (p1) circle (3pt) ;
 \filldraw (p2) circle (3pt);
   \filldraw (p3) circle (3pt);  
\filldraw  (q1) circle (3pt);
\filldraw [white]  (q1) circle (2pt);

\path (3.5, 6) coordinate (p4);
 \filldraw (p4) circle (3pt);  
 \path (3, 2) coordinate (p5);
  \filldraw (p5) circle (3pt);  
 \draw (p1) -- (q1);
  \draw (p2) -- (q1);
\draw (p3) -- (q1);
 \path (7.5, 2) coordinate (q2);
 \filldraw (q2) circle (3pt);
\filldraw (q3) circle (3pt);
\filldraw (q4) circle (3pt);
\filldraw (q5) circle (3pt);
\filldraw (q6) circle (3pt);

\draw (q3)--(p4)--(q4)--(p3);
 \draw (p2)--(q5)--(p5)--(q6);  
 \draw (p1) --(q2);
 \draw [red] (4,3) --(5.1,3) --(6, 2.7) -- (6,2) ;
\draw [red] (4,3) --(4.9,2.7) --(6, 2.7) ;
\draw [red] (4,3) --(4,2.3) --(4.9, 2) -- (6,2) ;
\draw [red] (4,2.3) --(5.1, 2.3) -- (6,2) ;
\draw [red] (5.1, 2.3) -- (5.1,3) ;
\draw [red] (4.9, 2.7) -- (4.9,2) ;
\draw [red] (4, 3)--(4.5, 3.2) --(5, 4) -- (4.5, 3.8) -- cycle ;
\draw [red] (4, 3)--(3, 3.3) --(2,3) -- (3,2.7) -- cycle ;

 \draw [red] (3.5,6) --(4.6,6) --(5.5, 5.7) -- (5.5,5) ;
\draw [red] (3.5,6) --(4.4,5.7) --(5.5,5.7) ;
\draw [red] (3.5,6) --(3.5,5.3) --(4.4, 5) -- (5.5,5) ;
\draw [red] (3.5,5.3) --(4.6,5.3) -- (5.5,5) ;
\draw [red] (4.6, 5.3) -- (4.6,6) ;
\draw [red] (4.4, 5.7) -- (4.4,5) ;
\draw [red] (6,2)--(6.75, 2.5) --(7.5,2) -- (6.75, 1.5) -- cycle ;
\draw [red] (2.5,1)--(2.5, 1.5) --(3,2) -- (3,1.5) -- cycle ;

    \filldraw [white] (p1) circle (2pt);
\filldraw [white]  (p2) circle (2pt);
\filldraw [white] (p3) circle (2pt);
\filldraw [white] (p5) circle (2pt);
\filldraw [white] (p4) circle (2pt);
\filldraw [white]  (q1) circle (2pt);
 \filldraw [white] (q2) circle (2pt);
\filldraw  [white] (q3) circle (2pt);
\filldraw [white]  (q4) circle (2pt);
\filldraw [white]  (q5) circle (2pt);
\filldraw [white] (q6) circle (2pt); 

%{white  \filldraw (q2) circle (2pt);
\draw (5, 0) node {Fig 1};
\end{tikzpicture}
% \caption{M}\label{fig:1block}
\end{figure}
\eject
 
\begin{figure}[htbp]
\centering
\begin{tikzpicture}[scale=1.1]
 \path (6,2) coordinate (p1);
    \path (2,3) coordinate (p2);
    \path (5,4) coordinate (p3);
   \draw (4, 3) coordinate (q1);
   \draw (3, 5) coordinate (q3);
    \draw (5.5, 5) coordinate (q4);
    \draw (1.5, 2) coordinate (q5);
    \draw (2.5, 1) coordinate (q6);
     \filldraw (p1) circle (3pt) ;
 \filldraw (p2) circle (3pt);
   \filldraw (p3) circle (3pt);  
\filldraw  (q1) circle (3pt);
\filldraw [white]  (q1) circle (2pt);

\draw (3, 5.5) node {$^1$};
\draw (4.6, 5.6) node {$^3$};
\draw (5.1, 4.5) node {$^1$};
\draw (4.4, 3.6) node {$^2$};
\draw (5, 2.6) node {$^3$};
\draw (3, 3.1) node {$^2$};
\draw (2.25,2.1) node {$^1$};
\draw (6.75, 2.1) node {$^2$};
\draw (1.5, 2.5) node {$^1$};
\draw (2.6, 1.5) node {$^2$};

\path (3.5, 6) coordinate (p4);
 \filldraw (p4) circle (3pt);  
 \path (3, 2) coordinate (p5);
  \filldraw (p5) circle (3pt);  
 \draw (p1) -- (q1);
  \draw (p2) -- (q1);
\draw (p3) -- (q1);
 \path (7.5, 2) coordinate (q2);
 \filldraw (q2) circle (3pt);
\filldraw (q3) circle (3pt);
\filldraw (q4) circle (3pt);
\filldraw (q5) circle (3pt);
\filldraw (q6) circle (3pt);

\draw (q3)--(p4)--(q4)--(p3);
 \draw (p2)--(q5)--(p5)--(q6);  
 \draw (p1) --(q2);

    \filldraw [white] (p1) circle (2pt);
\filldraw [white]  (p2) circle (2pt);
\filldraw [white] (p3) circle (2pt);
\filldraw [white] (p5) circle (2pt);
\filldraw [white] (p4) circle (2pt);
\filldraw [white]  (q1) circle (2pt);
 \filldraw [white] (q2) circle (2pt);
\filldraw  [white] (q3) circle (2pt);
\filldraw [white]  (q4) circle (2pt);
\filldraw [white]  (q5) circle (2pt);
\filldraw [white] (q6) circle (2pt); 
\draw (5, 0) node {Fig 2};

%{white  \filldraw (q2) circle (2pt);

\end{tikzpicture}
% \caption{$G$-Subgraph of $X$ and structure tree}\label{fig:}
\end{figure}

%Suppose now that $o$ and $go$ lie in  adjacent corners.   By relabelling we assume that $o \in A^*\cap gB$ and $go \in A^*\cap gB^*$.

\end{document}